\documentclass[12pt]{amsart}
\usepackage{amsmath, amssymb,amsfonts, color,epsf}
\usepackage{mathabx}
\usepackage [cmtip,arrow]{xy}
\xyoption{all}
\usepackage {pb-diagram,pb-xy}
\usepackage{enumerate}
\usepackage{accents}
\usepackage[noautoscale]{youngtab}
\usepackage{subfigure}

\ifx\pdfpageheight\undefined
   \usepackage[dvips,colorlinks=true,linkcolor=blue,citecolor=red,%
      urlcolor=green]{hyperref}
   \usepackage[dvips]{graphicx}
   \makeatletter
   \edef\Gin@extensions{\Gin@extensions,.mps}
   \makeatother
\else
 \usepackage[pdftex]{graphicx}
   \usepackage[bookmarksopen=false,pdftex=true,breaklinks=true,%
      backref=page,pagebackref=true,plainpages=false,%
      hyperindex=true,pdfstartview=FitH,colorlinks=true,%
      pdfpagelabels=true,colorlinks=true,linkcolor=blue,%
      citecolor=red,urlcolor=green,hypertexnames=false%
      ]%
   {hyperref}
\fi
\usepackage{bm}

\usepackage{tikz-cd}
\makeatletter
\tikzset{
  column sep/.code=\def\pgfmatrixcolumnsep{\pgf@matrix@xscale*(#1)},
  row sep/.code   =\def\pgfmatrixrowsep{\pgf@matrix@yscale*(#1)},
  matrix xscale/.code=%
    \pgfmathsetmacro\pgf@matrix@xscale{\pgf@matrix@xscale*(#1)},
  matrix yscale/.code=%
    \pgfmathsetmacro\pgf@matrix@yscale{\pgf@matrix@yscale*(#1)},
  matrix scale/.style={/tikz/matrix xscale={#1},/tikz/matrix yscale={#1}}}
\def\pgf@matrix@xscale{1}
\def\pgf@matrix@yscale{1}
\makeatother

\usepackage{amsmath,amssymb,amsfonts}
\newtheorem{theorem}{Theorem}

\newtheorem{corollary}{Corollary}
\newtheorem{proposition}{Proposition}[section]

\newtheorem*{claim*}{Claim}

\newtheorem*{theorem*}{Theorem}
\newtheorem*{corollary*}{Corollary}
\theoremstyle{definition}
\newtheorem{definition}{Definition}[section]
\newtheorem{example}{Example}[section]

\newtheorem{notation}{Notation}[section]

\theoremstyle{remark}
\newtheorem{remark}{Remark}[section]

\theoremstyle{remark}
\newtheorem{problem}{Problem}

\newcommand {\hide}[1]{}

\newcommand {\junk}[1]{}
\newtheorem{algorithm}{\sc Algorithm}

\newcommand {\R} {\mathrm{R}}

\newcommand {\D}     {\mbox{\rm D}}

\newcommand {\Z}  {\mathbb{Z}}

\newcommand {\Q}         {\mathbb{Q}}

\newcommand {\RR} {{\mathcal R}}

\newcommand {\la}   {{\langle}}
\newcommand {\ra}   {{\rangle}}
\newcommand {\eps} {{\varepsilon}}
\newcommand {\E} {{\rm Ext}}

\newcommand{\card}{\mathrm{card}}
\newcommand{\rank}{\mathrm{rank}}

\def\addots{\mathinner{\mkern1mu
		\raise1pt\vbox{\kern7pt\hbox{.}}
		\mkern2mu\raise4pt\hbox{.}\mkern2mu
		\raise7pt\hbox{.}\mkern1mu}}

\newcommand{\HH}  {\mbox{\rm H}}

\newcommand{\level}{\mathrm{level}}

\newcommand{\Top}{\mathbf{Top}}

\usepackage{algorithm}
\usepackage{algpseudocode}
\algnewcommand\algorithmicinput{\textbf{Input:}}
\algnewcommand\INPUT{\item[\algorithmicinput]}
\algnewcommand\algorithmicoutput{\textbf{Output:}}
\algnewcommand\OUTPUT{\item[\algorithmicoutput]}

\algnewcommand\algorithmicproc{\textbf{Procedure:}}
\algnewcommand\PROCEDURE{\item[\algorithmicproc]}
\algnewcommand\algorithmiccomplexity{\textbf{Complexity:}}
\algnewcommand\COMPLEXITY{\item[\algorithmiccomplexity]}
\usepackage{multicol}
\usepackage{tikz-cd}
\newcommand{\tcyl}{\widetilde{\mathrm{{cyl}}}}
\newcommand{\cyl}{\mathrm{cyl}}

\newcommand{\Simp}{\mathbf{Simp}}

\makeatletter 
\renewcommand\p@enumii{}
\makeatother

\begin{document}

\title[Computing the homology functor on semi-algebraic maps]
	{
	Computing the homology functor on semi-algebraic maps and diagrams	
	}
	\author{Saugata Basu}
	\address{Department of Mathematics,
		Purdue University, West Lafayette, IN 47906, U.S.A.}
	\email{sbasu@math.purdue.edu}
	
	\author{Negin Karisani}
	\address{Department of Computer Science, 
		Purdue University, West Lafayette, IN 47906, U.S.A.}
	\email{nkarisan@cs.purdue.edu}

	\date{\textbf{\today}}
	\thanks{
		Basu was  partially supported by NSF grants
		CCF-1910441.}

\begin{abstract}
    Developing an algorithm for computing the Betti numbers of semi-algebraic sets with singly exponential complexity has been a holy grail in algorithmic semi-algebraic geometry and only partial results are known. In this paper we consider the more general problem of computing the image under the homology functor of a semi-algebraic map $f:X \rightarrow Y$ between closed and bounded semi-algebraic sets. For every fixed $\ell \geq 0$ we give an algorithm with singly exponential complexity that computes bases of the homology groups $\HH_i(X), \HH_i(Y)$ (with rational coefficients)  and a matrix with respect to these bases of the induced linear maps
    $\HH_i(f):\HH_i(X) \rightarrow \HH_i(Y), 0 \leq i \leq \ell$. We generalize this algorithm to more general (zigzag) diagrams of maps between closed and bounded semi-algebraic sets and give a singly exponential algorithm for computing 
    the homology functors on such diagrams.
    This allows us to give an algorithm with singly exponential
    complexity for computing barcodes of semi-algebraic zigzag persistent homology 
    in small dimensions.
\end{abstract}

	\maketitle
	
	\tableofcontents

\section{Introduction}
\label{sec:intro}
Let $\R$ be a real closed field and $\D$ an ordered domain contained in $\R$.

The problem of effective computation of topological properties of semi-algebraic subsets of $\R^k$ has a long
history.  Semi-algebraic subsets of $\R^k$ are subsets defined by first-order formulas in the language of ordered
fields (with parameters in $\R$). Since the first-order theory of real closed fields admits 
quantifier-elimination, we can assume that each semi-algebraic subset $S \subset \R^k$ is defined by some 
\emph{quantifier-free} formula $\phi$. A quantifier-free formula $\phi(X_1,\ldots,X_k)$ in the 
language of ordered fields with parameters in $\D$, is a formula with atoms of the form $P = 0, P > 0, P < 0$,
$P \in \D[X_1,\ldots,X_k]$.

Semi-algebraic subsets of $\R^k$ have tame topology. In particular, closed and bounded semi-algebraic subsets
of $\R^k$ are semi-algebraically triangulable (see for example \cite[Chapter 5]{BPRbook2}). 
This means that there exists a finite simplicial complex $K$, whose 
\emph{geometric realization}, $|K|$,  considered as a subset of $\R^N$ for some $N >0$, is semi-algebraically homeomorphic to $S$. 
The semi-algebraic homeomorphism
$|K| \rightarrow S$ is called a \emph{semi-algebraic triangulation of $S$}. All topological
properties of $S$ are then encoded in the finite data of the simplicial complex $K$. 

There exists a classical algorithm
which takes as input a quantifier-free formula defining a semi-algebraic set $S$, and 
produces as output a semi-algebraic triangulation of $S$  (see for instance \cite[Chapter 5]{BPRbook2}).  
However,  this algorithm is based on the technique of \emph{cylindrical algebraic decomposition}, and hence 
the complexity of this algorithm is prohibitively expensive,  being doubly exponential in $k$. 
More precisely, given a description by a quantifier-free formula involving 
$s$ polynomials of degree at most $d$, of a closed and bounded semi-algebraic subset of $S \subset \R^k$, there
exists an algorithm computing a semi-algebraic triangulation of $h: |K| \rightarrow S$, whose complexity
is bounded by $(s d)^{2^{O(k)}}$. Moreover, the size of the simplicial complex $K$ (measured by the number of 
simplices) is also bounded by $(s d)^{2^{O(k)}}$.

\subsection{Doubly vs singly exponential} 
One can ask whether the doubly exponential behavior for the semi-algebraic triangulation problem is intrinsic
to the problem. One reason to think that it is not so comes from the fact that the ranks of the homology groups of 
$S$ (following the same notation as in the previous paragraph), 
and so in particular those of the simplicial complex $K$, is bounded by
$(O(sd))^k$ (see for instance \cite[Chapter 7]{BPRbook2}), 
which is singly exponential in $k$. So it is natural to ask if this singly exponential upper bound
on $\rank(\HH_*(S))$ is ``witnessed''  by an efficient semi-algebraic triangulation of small (i.e. singly exponential) size.
This is not known.

In fact, designing an algorithm with a \emph{singly exponential} complexity 
for computing a semi-algebraic triangulation of a given semi-algebraic set 
has remained a holy grail in the field of algorithmic real algebraic geometry and little progress has been made over the last 
thirty years on this problem (at least for general semi-algebraic sets). 
We note here that designing algorithms with singly exponential complexity has being a \emph{leit motif}  in the research in algorithmic semi-algebraic geometry over the past decades -- starting from the so called ``critical-point method'' which resulted in algorithms for testing emptiness, connectivity, computing the Euler-Poincar\'e characteristic, as well as for the first few Betti numbers of semi-algebraic sets (see \cite{Basu-survey} for a history of these developments and contributions of many authors). More recently, such algorithms have also been 
developed in other (more numerical) models of computations \cite{BCL2019,BCT2020.1, BCT2020.2}.

In \cite{BCL2019, BCT2020.1, BCT2020.2}, the authors take a different approach. Working over $\mathbb{R}$, and given a 
``well-conditioned''  semi-algebraic subset $S\subset \mathbb{R}^k$,  
they compute a witness complex whose geometric realization is $k$-equivalent to $S$. The size of this witness
complex as well the complexity of the algorithm is bounded singly exponentially in $k$, but also depends on a real parameter, namely the  condition number of the input 
 (and so this bound is not uniform). The algorithm will fail for ill-conditioned input when the condition number becomes
 infinite. This is unlike the kind of algorithms we consider in the current paper, which are supposed to work for all inputs
 and with uniform complexity upper bounds. So these approaches are not comparable.

\subsection{Homology as a functor}
Homology is a functor from the category of topological spaces
to $\mathbb{Z}$-modules. Restricted to the category of semi-algebraic sets and maps and considering homology groups with only rational coefficients, it is a functor from
the category of semi-algebraic subsets of $\R^k, k >0$ to finite dimensional
$\Q$-vector spaces.
The algorithms discussed in the previous section aimed only at
computing the dimension of the homology groups. 
However,
a very natural algorithmic question that arises
is the following.

\begin{problem}
\label{prob:1}
Given a first-order formula $\phi_f$ describing the graph of a semi-algebraic map $f:X \rightarrow Y$, compute with singly exponential complexity a description of the
map $\HH_i(f): \HH_i(X) \rightarrow \HH_i(Y)$ (i.e. compute a basis of 
of $\HH_i(X),\HH_i(Y)$ and the matrix corresponding to these bases of the
linear map $\HH_i(f)$). More generally, given a diagram of semi-algebraic maps, compute with singly exponential complexity bases 
of the homology groups of the various semi-algebraic sets and matrices corresponding to the 
different maps. We will say that such an algorithm \emph{computes
the homology functor for semi-algebraic maps (or more generally diagram of maps) in dimension $i$}.
\end{problem}

\begin{remark}
\label{rem:categorical}
Studying the ``functor complexity'' of the homology functor was raised in 
\cite[Section 7, Problem (4)]{Basu-Isik} in the setting of categorical complexity. In this paper we initiate the study of this functor from the complexity point of view, though the definition of complexity that we use
in this paper is the classical notion and not the categorical one introduced in \cite{Basu-Isik}. 
\end{remark}

\begin{remark}
\label{rem:sa-maps-not-triangulable}
One important point to note that is that semi-algebraic maps 
$f:X \rightarrow Y$ between closed and bounded semi-algebraic sets 
are not necessarily triangulable (unless $\dim Y \leq 1$). 
An easy example is the so called ``blow-down'' map.

\begin{example}
\label{eg:blow-down}
Let $S \subset \R^3$ be defined by the formula
\[
(Y - ZX  = 0) \wedge (X^2 + Y^2 - 1 \leq 0) \wedge (Y - X \leq 0) \wedge (X \geq 0),
\]
$T$ the unit disk in $\R^2$, and $f:S \rightarrow T$
the projection map along the $Z$-coordinate. The map $f$ is easily seen to be not triangulable. More precisely, there are no semi-algebraic triangulations 
\[
h_S:|K_S| \rightarrow S, h_T: |K_T| \rightarrow T,
\]
and a simplicial map $F:K_S \rightarrow K_T$, such
that 
\[
|F| \circ h_S = h_T \circ F.
\]
\end{example}
Thus, one cannot expect to solve Problem~\ref{prob:1} by
computing semi-algebraic triangulations of $h_X: |K_X| \rightarrow X$ and $h_Y:|K_Y| \rightarrow Y$, such that the 
induced map $h_Y^{-1} \circ f \circ h_X$ is simplicial.
\end{remark}

\begin{remark}
It is not at all clear if the algorithms designed so far  for computing
Betti numbers of semi-algebraic sets with singly exponential complexity
(both exact algorithms such as those in 
\cite{Bas05-first,BPRbettione} or numerical ones such as those in
\cite{BCL2019, BCT2020.1, BCT2020.2})
can extend to solve Problem~\ref{prob:1}.
\end{remark}

\subsection{Main contributions}
The main contribution of this paper is a partial solution to Problem~\ref{prob:1}. 
We prove the following theorem which we state informally here
(see Theorem~\ref{thm:functor} in Section~\ref{sec:main-results} for a precise statement).

\begin{theorem*}[Computing homology functor on semi-algebraic maps]
For each fixed $\ell \geq 0$,
there exists an algorithm with singly exponential complexity
that computes the homology functor for semi-algebraic maps  between closed and bounded semi-algebraic sets in each dimension $i, 0 \leq i \leq \ell$.
\end{theorem*}

\begin{remark}
\label{rem:quivers}
Note that up to isomorphism (in the category of vector spaces)  a linear map $L: V \rightarrow W$ between finite dimensional vector spaces $V,W$ is determined by the numbers $\dim V, \dim W, \mathrm{rank}(L)$. Thus, for a 
semi-algebraic map $f:X \rightarrow Y$, computing a description
up to isomorphism of the linear map $\HH_i(f):\HH_i(X) \rightarrow \HH_i(Y)$, amounts to computing $\dim \HH_i(X),\dim \HH_i(Y), \rank(\HH_i(f))$. However, the isomorphism class of more general
diagrams of vector spaces (such as zigzag diagrams in Theorem~\ref{thm:zigzag}) is not determined just by the dimensions of the vector spaces and the ranks of the linear maps. 
\end{remark}

\subsubsection{More general diagrams}
Once we have an algorithm for computing the homology functor on semi-algebraic maps it is natural to try to extend it to more complicated diagrams of maps. As an example,
	in this paper we consider \emph{zigzag diagrams} (see Notation~\ref{not:zigzag} below) 
	of semi-algebraic maps.
	We prove the following theorem (see Theorem~\ref{thm:zigzag} for a more precise statement).
	
	\begin{theorem*}[Computing homology functor on  zigzag diagrams]
	For each fixed $\ell \geq 0$,
there exists an algorithm with singly exponential complexity
that computes the homology functor for diagrams of semi-algebraic maps of the zigzag type 
between closed and bounded semi-algebraic sets in each dimension $i, 0 \leq i \leq \ell$.
    \end{theorem*}

\subsubsection{Computing semi-algebraic zigzag persistence}
As an application of the previous theorem we consider the problem
of computing zigzag persistence modules for semi-algebraic maps. 
Persistent homology theory is foundational in the area of topological data analysis mainly in the context of finite simplicial complexes (see for example \cite{Edelsbrunner-Harer,Dey-Wang2021} for background and many applications of persistent homology theory).  

Persistence homology is defined for any filtration of topological spaces 
and its underlying module structure gives rise to barcodes via decomposition into irreducibles.  
The algorithmic study of persistent homology for filtrations of semi-algebraic sets by the sublevel sets of a polynomial function was initiated in \cite{basu-karisani-persistent} and we refer the reader to that paper for the basic definitions including that of barcodes.
Although persistent homology was originally defined for filtrations
of topological spaces, it
has since been generalized to arbitrary diagrams 
$D:P \rightarrow \mathbf{Top}$ (here $D$ is a functor from a poset category $P$  to the category of topological spaces) \cite{Bubenik-et-al}. 
	One particular class of diagrams that has been studied in the literature are zigzag diagrams.
	Zigzag persistence modules was introduced in \cite{CarlssonSilva2010} and was studied from the algebraic as well as
	algorithmic point of view. In particular, they showed that it is possible to
	associate barcodes to zigzag persistent homology as well and gave an algorithm
	to compute them. Zigzag persistence is a very active area of current research \cite{Botnan-et-al, Carlsson-et-al}.
	
	The previous theorem allows to reduce the computation of the barcode of a zigzag diagram
	of semi-algebraic maps to that of finite dimensional vector spaces, and
	here we can use the algorithm in \cite[Section 4.2]{CarlssonSilva2010}.
	In this way we obtain for each fixed $\ell$, a singly exponential algorithm for computing zigzag persistent module for semi-algebraic zigzag diagrams in dimensions $0$ to $\ell$.
We have the following theorem.

\begin{theorem}[Computing semi-algebraic zigzag persistent modules]
\label{thm:zigzag-persistence}
For each fixed $\ell \geq 0$,
there exists an algorithm with singly exponential complexity
that computes the barcode of diagrams of semi-algebraic maps of the zigzag type 
between closed and bounded semi-algebraic sets in each dimension $i, 0 \leq i \leq \ell$.
\end{theorem}

\subsection{Key ideas}
\label{subsec:key-ideas}
We summarize here the key ideas that go into the proofs of the theorems stated above.
\subsubsection{Replacing semi-algebraic maps by inclusions which are homologically equivalent}
Semi-algebraic maps between closed and bounded semi-algebraic sets which are inclusions can be treated much more easily than in the general case. Indeed,  the main result proved in \cite{basu-karisani} implies that for each fixed $\ell \geq 0$, there exists an algorithm that given a semi-algebraic inclusion map
$X \hookrightarrow Y$ between closed and bounded semi-algebraic sets
$X,Y$ as input, computes two finite simplicial complexes, $K$ and $L$
with $K \subset L$, such that the inclusion of the corresponding geometric realizations $|K| \hookrightarrow |L|$ is $\ell$-equivalent to the inclusion $X \hookrightarrow Y$ (see Definition~\ref{def:ell-equivalent}). The key new idea in this paper is to replace an
arbitrary semi-algebraic map between closed and bounded semi-algebraic sets, by an inclusion map which is equivalent to the given map up to homotopy. 

\subsubsection{Realizing the mapping cylinder of a semi-algebraic map up to homotopy as a semi-algebraic set}
The standard tool for achieving this involves taking the mapping cylinder, $\cyl(f)$, of the map $f$. The canonical inclusion $X \hookrightarrow \cyl(f)$
is then homotopy equivalent to the $f$. 
However, the definition of mapping cylinder (see Definition~\ref{def:cylinder-top} below)
involves identification of certain points of a disjoint union (or equivalently passing to a quotient space). While quotients
of semi-algebraic sets by proper equivalence relations are semi-algebraic (see for example \cite{Dries}),
no singly exponential algorithm for computing a semi-algebraic description of such a quotient is known. 

We overcome this difficulty by modifying the construction of the mapping cylinder (see Section~\ref{subsec:inclusion}).
We associated to the map $f$, a modified mapping cylinder of $f$,
which we denote $\tcyl(f)$ (see Definition~\ref{def:cylinder}) which is a semi-algebraic set having similar properties as mapping cylinder of $f$ (see Proposition~\ref{prop:tcyl}). 

The definition of $\tcyl(f)$ does not involve taking quotients.  However, it does involve quantifier elimination of an existential block of quantifiers (or equivalently taking the image under a linear projection map).
This leads to a further technical complication.
It is important for us in order to be able to apply the 
result of \cite{basu-karisani} that the semi-algebraic sets that we deal are not only closed, but
are described by closed formulas (see Notation~\ref{not:sign-condition}). 
While the image under projection of a closed and bounded
semi-algebraic set is closed and bounded, the quantifier elimination algorithm that we use to obtain its description by a quantifier-free formula is not guaranteed to produce a closed description. Indeed no algorithm with singly exponential complexity is known for obtaining a closed description of a
given closed semi-algebraic set and designing such an algorithm is considered to be
a difficult open problem in algorithmic semi-algebraic geometry.
Thus, we need an additional step. 

\subsubsection{Replacing closed semi-algebraic set by ones described by closed formulas}
We replace
(see Section~\ref{subsec:making-closed})
a closed semi-algebraic set by another
one, which is infinitesimally larger but has the same homotopy type, and  
moreover is described by a closed formula having size bounded linearly in the size of the original formula (see Notation~\ref{not:phi-star} and Proposition~\ref{prop:closed} below). For this purpose, as usual in algorithmic semi-algebraic geometry we utilize extensions (obtained by adjoining infinisteimal elements)
of the given real closed fields by fields of Puiseux series in these 
infinitesimals (see Section~\ref{subsec:Puiseux}). 

\subsubsection{Mapping cylinder for diagrams}
Finally, to extend our algorithm to zigzag diagrams we need
to further generalize the definition of $\tcyl(f)$, so that every map
in the diagram \emph{simultaneously} becomes inclusions without changing the homological type of the diagram (see Section~\ref{sec:zigzag}).  
We generalize the definition of $\tcyl(f)$ to define $\tcyl(D)$, the semi-algebraic mapping cylinder of a zigzag diagram $D$
(Definition~\ref{def:cylinder-general}).
We prove that the $\tcyl(D)$ and $D$ are homologically equivalent (Proposition~\ref{prop:zigzag}).
Using similar techniques as in the case of maps (i.e. replacing by a set defined by closed formulas etc.)
we are then able to extend the algorithm for maps to the case of zigzag diagrams, and ultimately give an algorithm to compute barcodes of semi-algebraic zigzag diagrams. \\

The rest of the paper is organized as follows. In Section~\ref{sec:prelim} we fix notation
and give precise definitions of complexity and topological equivalences. We also give the necessary background in real algebraic geometry to make the rest of the paper self-contained.
In Section~\ref{sec:main-results} we give precise statements of the theorems proved in this paper. The subsequent sections are devoted to the proofs of these theorems.
In Section~\ref{sec:prelim-math} we state and prove some mathematical results that
play an important role in the algorithms described in this paper.
In Section~\ref{subsec:inclusion}, we give the construction of the
semi-algebraic mapping cylinder (i.e. of the semi-algebraic set $\tcyl(f)$ referred to in the previous paragraph) and prove its main properties.
In Section~\ref{subsec:making-closed} we give the procedure for replacing a given closed semi-algebraic set by one having the same homotopy type and which is described by a closed formula.
In Section~\ref{sec:proof} we complete the 
proof of Theorem~\ref{thm:functor}. Finally, in Section~\ref{sec:zigzag} we apply the ideas 
developed in the proof of Theorem~\ref{thm:functor} to develop an algorithm for computing 
semi-algebraic zigzag persistent barcodes. In Section~\ref{sec:conclusion} we state some open problems.

\section{Preliminaries}
\label{sec:prelim}
\subsection{Homological equivalence of semi-algebraic maps}
We begin with the precise definitions of the two kinds of 
topological equivalence that we are going to use in this paper.

\subsubsection{Homological equivalences}
\begin{definition}[Homological $\ell$-equivalences]
\label{def:equivalence-spaces}
We say that a semi-algebraic map $f:X \rightarrow Y$ between two 
semi-algebraic sets $X,Y$ is a homological $\ell$-equivalence,
if the induced homomorphisms between the homology groups 
$\HH_i(f): \HH_i(X) \rightarrow \HH_i(Y)$ are isomorphisms for $0 \leq i \leq \ell$.

Given two semi-algebraic maps $f:X \rightarrow Y,f': X' \rightarrow Y'$,
a homological $\ell$-equivalence between $f$ and $f'$ is a pair of semi-algebraic maps $F_X, F_Y$ such that
$f' \circ F_X = F_Y \circ f$, and $F_X, F_Y$ are homological $\ell$-equivalences. 
\end{definition}

The relation of homological $\ell$-equivalence as defined above is not an equivalence relation since it is not 
symmetric. In order to make it symmetric one needs to ``formally invert'' homologically $\ell$-equivalences.

\begin{definition}[Homologically $\ell$-equivalent]
\label{def:ell-equivalent}
We will say that  \emph{$X$ is homologically $\ell$-equivalent to $Y$}  (denoted $X \sim_\ell Y$), if and only if there exists 
spaces, $X=X_0,X_1,\ldots,X_n=Y$ and homological $\ell$-equivalences  $f_1,\ldots,f_{n}$ as shown below:
\[
\xymatrix{
&X_1 \ar[ld]_{f_1}\ar[rd]^{f_2} &&X_3\ar[ld]_{f_3} \ar[rd]^{f_4}& \cdots&\cdots&X_{n-1}\ar[ld]_{f_{n-1}}\ar[rd]^{f_{n}} & \\
X_0 &&X_2  && \cdots&\cdots &&  X_n&
}.
\]

Similarly, we say that a semi-algebraic map $f:X \rightarrow X'$ is homologically $\ell$-equivalent to the semi-algebraic map $g:Y \rightarrow Y'$, if there exists maps $f = f_0, f_1,\ldots, f_n = g$, and homological
$\ell$-equivalences, $F_1,\ldots,F_n$, as below:
\[
\xymatrix{
&f_1 \ar[ld]_{F_1}\ar[rd]^{F_2} &&f_3\ar[ld]_{F_3} \ar[rd]^{F_4}& \cdots&\cdots&f_{n-1}\ar[ld]_{F_{n-1}}\ar[rd]^{F_{n}} & \\
f_0 &&f_2  && \cdots&\cdots &&  f_n&
}.
\]

It is clear that $\sim_\ell$ is an equivalence relation.
\end{definition}

 \begin{remark}
 \label{rem:homology-vs-homotopy}
 One main tool that we use is the Vietoris-Begle theorem. Since, there are many versions of the Vietoris-Begle theorem in the literature we make precise what we use below.
 If
 $X \subset \mathbb{R}^m, Y \subset \mathbb{R}^n$ are compact semi-algebraic subsets (and so are locally contractible), and 
 $f:X \rightarrow Y$ is a semi-algebraic continuous map such that 
  $f^{-1}(y)$ is homologically $\ell$-connected for each $y \in Y$, 
  then we can conclude that
 $f$ is a homological $\ell$-equivalence (see for example, the statement
 of the  Vietoris-Begle theorem in \cite{Ferry2018}). This latter
 theorem is also valid 
 for semi-algebraic maps between closed and bounded semi-algebraic sets
 over arbitrary real closed fields, once we know it for maps between compact semi-algebraic subsets over $\mathbb{R}$. This follows from a standard argument using the  Tarski-Seidenberg transfer principle and the fact that homology groups of closed bounded semi-algebraic sets can be defined in terms of finite triangulations. We will refer to this version of the Vietoris-Begle theorem  as
 the \emph{homological version of the Vietoris-Begle theorem}.
 \end{remark}

\subsection{Definition of complexity of algorithms}
We will use the following notion of ``complexity of an algorithm''  in this paper. We follow the same definition as used in the book \cite{BPRbook2}. 
  
\begin{definition}[Complexity of algorithms]
\label{def:complexity}
In our algorithms we will take as input quantifier-free first order formulas whose terms
are  polynomials with coefficients belonging to an ordered domain $\D$ contained in a real closed field $\R$.
By \emph{complexity of an algorithm}  we will mean the number of arithmetic operations and comparisons in the domain $\D$.
If $\D = \mathbb{R}$, then
the complexity of our algorithm will agree with the  Blum-Shub-Smale notion of real number complexity \cite{BSS}.
In case, $\D = \Z$, then we are able to deduce the bit-complexity of our algorithms in terms of the bit-sizes of the coefficients
of the input polynomials, and this will agree with the classical (Turing) notion of complexity.
\end{definition}

\subsection{Real algebraic preliminaries}
\begin{notation}[Realizations, $\mathcal{P}$-, $\mathcal{P}$-closed
semi-algebraic sets]
  \label{not:sign-condition} 
  For any finite set of polynomials $\mathcal{P}
  \subset \R [ X_{1} , \ldots ,X_{k} ]$, 
  we call any quantifier-free first order formula $\phi$ with atoms, $P =0, P < 0, P>0, P \in \mathcal{P}$, to
  be a \emph{$\mathcal{P}$-formula}. 
  
  Given any semi-algebraic subset $Z \subset \R^k$,
  we call the realization of $\phi$ in $Z$,
  namely the semi-algebraic set
  \begin{eqnarray*}
    \RR(\phi,Z) & := & \{ \mathbf{x} \in Z \mid
    \phi (\mathbf{x})\}
  \end{eqnarray*}
  a \emph{$\mathcal{P}$-semi-algebraic subset of $Z$}.
  
  If $Z = \R^k$, we will denote the realization of $\phi$ in $\R^k$ by
  $\RR(\phi)$.
  
 We say that a quantifier-free formula $\phi$ is \emph{closed}  
  if it is a formula in disjunctive normal form with no negations, and with atoms of the form $P \geq 0, P \leq 0$ (resp. $P > 0, P < 0$),  
where $P \in \R[X_1,\ldots,X_k]$. If the set of polynomials appearing in a closed formula
is contained in a finite set $\mathcal{P}$, we will call such a formula a $\mathcal{P}$-closed 
formula,
 and we call
  the realization, $\RR \left(\phi \right)$, a \emph{$\mathcal{P}$-closed 
  semi-algebraic set}.
  \end{notation}

We now state precisely the main results proved in this paper.
\section{Main Results}
\label{sec:main-results}

\begin{theorem}
\label{thm:functor2}
For each $\ell \geq 0$, there is an algorithm that accepts as input
\begin{enumerate}[(a)]
    \item finite sets of polynomials 
    \[
    \mathcal{P}_S \subset \D[X_1,\ldots,X_k],
    \]
    \[
    \mathcal{P}_T \subset \D[Y_1,\ldots,Y_m],
    \]
    \[
    \mathcal{P}_f \subset \D[X_1,\ldots,X_k,Y_1,\ldots Y_m];
    \]
    \item
    a $\mathcal{P}_S$-closed formula $\phi_S$, 
a $\mathcal{P}_T$-closed formula $\phi_T$,
and 
a $\mathcal{P}_f$-closed formula $\phi_f$,
such that
$\RR(\phi_S), \RR(\phi_T)$ are bounded and 
$\RR(\phi_f,\R^k \times \R^m)$ is the graph of a semi-algebraic map
$f: S = \RR(\phi_S) \rightarrow  \RR(\phi_T) = T$;
\end{enumerate}
and produces as output for each $i, 0 \leq i \leq \ell$:
\begin{enumerate}[(a)]
\item
bases  
of $\HH_i(S),\HH_i(T)$;
\item
the matrix corresponding to these bases of the
linear map $\HH_i(f): \HH_i(X)\rightarrow \HH_i(Y)$.
\end{enumerate}

The complexity of the algorithm is bounded by 
\[
(s d)^{(k+m)^{O(\ell)}},
\]
where
\[
s = \max(\card(\mathcal{P}_S), \card(\mathcal{P}_T), \card(\mathcal{P}_f)),
\]
and
\[
d = \max_{P \in \mathcal{P}_S \cup \mathcal{P}_T \cup \mathcal{P}_f} \deg(P).
\]
\end{theorem}

Theorem~\ref{thm:functor2} will follow (using standard linear
algebra algorithms) from the following theorem.

\begin{theorem}
\label{thm:functor}
For each $\ell \geq 0$, there is an algorithm that accepts as input
\begin{enumerate}[(a)]
    \item finite sets of polynomials 
    \[
    \mathcal{P}_S \subset \D[X_1,\ldots,X_k], 
    \]
    \[
    \mathcal{P}_T \subset \D[Y_1,\ldots,Y_m],
    \]
    \[
    \mathcal{P}_f \subset \D[X_1,\ldots,X_k,Y_1,\ldots Y_m];
    \]
    \item
    a $\mathcal{P}_S$-closed formula $\phi_S$, 
a $\mathcal{P}_T$-closed formula $\phi_T$,
and 
a $\mathcal{P}_f$-closed formula $\phi_f$,
such that
$\RR(\phi_S), \RR(\phi_T)$ are bounded and 
$\RR(\phi_f,\R^k \times \R^m)$ is the graph of a semi-algebraic map
$f: S = \RR(\phi_S) \rightarrow  \RR(\phi_T) = T$;
\end{enumerate}
and produces as output  simplicial complexes $\Delta_S, \Delta_T, \Delta_S \subset \Delta_T$, such that
$|\Delta_S| \hookrightarrow |\Delta_Y|$ is homologically 
$\ell$-equivalent to $f: S \rightarrow T$.

The complexity of the algorithm is bounded by 
\[
(s d)^{(k+m)^{O(\ell)}},
\]
where
\[
s = \max(\card(\mathcal{P}_S), \card(\mathcal{P}_T), \card(\mathcal{P}_f)),
\]
and
\[
d = \max_{P \in \mathcal{P}_S \cup \mathcal{P}_T \cup \mathcal{P}_f} \deg(P).
\]
\end{theorem}

We extend Theorem~\ref{thm:functor2} to zigzag diagrams. We prove the following theorem.

\begin{theorem}
\label{thm:zigzag2}
    For each fixed $\ell \geq 0$, there exists an algorithm with the
    following properties. The algorithm takes the following input:
    \begin{enumerate}[1.]
    \item $R >0$;
        \item 
        a tuple of closed formulas
    $\Phi = (\phi_0,\ldots,\phi_n)$, with $S_i = \RR(\phi_i,B) \subset \R^k$ for $0 \leq i \leq n$, where $B = \overline{B_k(0,R)}$;
    \item
    a tuple of 
    closed formulas 
    $
    \Psi = (\psi_1,\ldots,\psi_n),
    $
    such that $\RR(\Psi_i, B \times B)$ is the graph of 
    a semi-algebraically continuous map $f_i:S_i \rightarrow S_{i-1}$
    if $i$ is odd, and is the graph of 
    a semi-algebraically continuous map $f_i:S_{i-1} \rightarrow S_{i}$ if $i$ is even.
    \end{enumerate}
    The algorithm produces as output for each $i,0 \leq i \leq \ell$:
    \begin{enumerate}[(a)]
        \item 
        Bases of the homology groups $\HH_i(S_0), 1 \leq j \leq n$;
        \item
        Matrices of the maps $\HH_i(f_j): \HH_i(S_{j-1}) \rightarrow \HH_i(S_j), 1 \leq j \leq n$.
    \end{enumerate}
     
    The complexity of the algorithm is bounded by $(n s d)^{k^{O(\ell)}}$,
where $s$ is the cardinality of the set of polynomials occurring in all the formulas in the input and $d$ their maximum degree.
\end{theorem}

\section{Mathematical Preliminaries}
\label{sec:prelim-math}
In this section we state and prove some mathematical results that will play a role in
the proofs of the main theorems.

\subsection{Replacing semi-algebraic maps by inclusion maps}
\label{subsec:inclusion}
The key idea that goes into the proof of Theorem~\ref{thm:functor} below is a semi-algebraic adaptation of the classical topological notion of a
mapping cylinder of a map $f:X \rightarrow Y$ which we recall now.

\begin{definition}[Mapping cylinder]
\label{def:cylinder-top}
Given a map $f: X \rightarrow Y$, the mapping cylinder $\cyl(f)$ of $f$ is the space defined by
\[
\cyl(f) = \left((X \times [0,1]) \coprod Y\right)/\sim,
\]
where for each $x \in X$, $(x,0) \sim f(x)$. 
\end{definition}

It is easy to prove that there exists a deformation retraction
$p:\cyl(f) \rightarrow Y$.
Denoting by $i:X \rightarrow \cyl(f)$, the inclusion $i(x) = (x,1)$,
we have a factorization $f = p \circ i$. Since, $p$ is a homotopy equivalence, one obtains that the inclusion $i:X \rightarrow \cyl(f)$
is homologically equivalent to the map $f$ via the commutative diagram
\[
\xymatrix{
X \ar[r]^i \ar[d]^{\mathrm{id}} & \cyl(f) \ar[d]^p \\
X \ar[r]^f & Y
}.
\]

Now suppose that $f:S \rightarrow T$ is a semi-algebraic map between
two closed and bounded semi-algebraic sets $S,T$. The mapping cylinder 
construction suggests a way to construct an inclusion map 
$i:S \hookrightarrow \cyl(f)$ which is homologically equivalent to $f$.
This is important for us since once we have replaced the given map $f$ by an inclusion, we can apply the main result in \cite{basu-karisani} to obtain a pair of simplicial complexes, $\Delta_1 \subset \Delta_2$ such that  the inclusion $|\Delta_1| \hookrightarrow |\Delta_1|$ is homologically $\ell$-equivalent to $i:S \hookrightarrow \cyl(f)$, and 
hence to $f$.

However, one obstruction to realizing the above goal is the fact that
the definition of the mapping cylinder involves taking a quotient. It is true that the quotient of a closed and bounded semi-algebraic set by
a proper equivalence relation is homeomorphic to a semi-algebraic set \cite{Dries} -- however, there is no algorithm known with a singly exponential complexity
for obtaining a semi-algebraic description of this quotient. 

We take a slightly different route. We define below a modification
of the classical mapping cylinder of a map $f$, which we denote by 
$\tcyl(f)$ (see Definition~\ref{def:cylinder})
which in the case where $f$ is a semi-algebraic map between closed and bounded semi-algebraic sets satisfies the same properties
as the classical mapping cylinder i.e. $f$ factorizes through an inclusion $i:S \rightarrow \tcyl(f)$ and a semi-algebraic homotopy equivalence $p:\tcyl(f) \rightarrow T$, so that finally $f = p \circ i$, and the inclusion $i:S \rightarrow \tcyl(f)$ is homologically equivalent to $f$. The main advantage of $\tcyl(f)$ is that, as a semi-algebraic set
it is described by a (existentially) quantified formula 
(see Eqn.\eqref{eqn:formula}) which is 
determined in a simple way from any first-order formulas defining the semi-algebraic sets $S$, $T$ and the graph of the map $f$. Using
effective quantifier-elimination algorithms we can then obtain
a quantifier-free formula defining $\tcyl(f)$.

There is one technical issue that creates complications in the 
above picture. If $S,T$ are closed and bounded semi-algebraic sets,
then the semi-algebraic set $\tcyl(f)$ is obtained as the image under projection of a closed and bounded semi-algebraic set, and is thus known
to be closed and bounded. However, even if we start with closed formulas
defining $S,T$ and $\mathrm{graph}(f)$, since the known effective quantifier-elimination algorithm with single exponential complexity that we use does not guarantee that the quantifier-free formula that we
obtain describing $\tcyl(f)$ is closed. It is important for the 
algorithm downstream that we use for simplicial replacement that 
this description be closed. We deal with this technical issue in a subsequent section. 

We now define $\tcyl(f)$.

\begin{definition}[Mapping cylinder for semi-algebraic maps]
\label{def:cylinder}
Let $S \subset \R^k, T \subset \R^m$ be semi-algebraic subsets and 
    $f:S \rightarrow T$ a semi-algebraic map.
    We denote
    \begin{equation}\label{eqn:tcyl}
				\tcyl(f) = \{ (\lambda \cdot x, f(x), \lambda) \ | \ x \in S, \lambda \in [0, 1]\} \cup \{ (\mathbf{0}, y, 0) \ | \ y \in T\},
			\end{equation}
\end{definition}			

With the above notation we have the following proposition.
\begin{proposition}
\label{prop:homotopy-equivalent}
Suppose that $S$ is closed and bounded and 
let $r: \tcyl(f) \rightarrow T$ be the map defined by
$r(x, y, \lambda) = y$. Then, $r$ is a homological equivalence.
\end{proposition}

\begin{proof}
It follows from Eqn. \eqref{eqn:tcyl}, that
for $y \in T$, 
\[
r^{-1}(y) = \{ (\lambda \cdot x, y , \lambda) \ | \ x \in S, f(x) = y, \lambda \in [0, 1]\} \cup \{(\mathbf{0}, y, 0)\}.
\]
There are two cases.
\begin{enumerate}
\item
If $y \in \mathrm{Im}(f)$, then
\[
r^{-1}(y) = \{ (\lambda \cdot x, y , \lambda) \ | \ x \in S, f(x) = y, \lambda \in [0, 1]\},
\]
which is semi-algebraically homeomorphic to the cone over $f^{-1}(y)$
and hence semi-algebraically contractible.
\item
If $y \not\in \mathrm{Im}(f)$, then
\[
r^{-1}(y) = \{(\mathbf{0}, y, 0)\}
\]
and hence semi-algebraically contractible.
\end{enumerate}
The proposition now follows from the homological version of the Vietoris-Begle theorem 
(see Remark~\ref{rem:homology-vs-homotopy}).
\end{proof}		

\begin{proposition}
\label{prop:injective}
Let 
\[
i: S \rightarrow \tcyl(f)  
\]
be the injective map $x \mapsto(x, f(x), 1)$. 
Then the following diagram is commutative.

\begin{equation}\label{diag:cyl}
				\begin{tikzcd}
					S \arrow[r, "f"] \arrow[d, hook, "i"] & T \arrow[d, "id"] \\
					\tcyl(f)\arrow[r, "r"]& T
				\end{tikzcd}
\end{equation}
\end{proposition}	

\begin{proof}
This is immediate from the definition of $\tcyl(f)$ (Eqn. \eqref{eqn:tcyl})
and the definition of the map $r$.
\end{proof}

\begin{proposition}
\label{prop:tcyl}
Suppose that $S$ is closed and bounded. Then the inclusion map,
$i(S) \hookrightarrow \tcyl(f)$ is homologically equivalent to
$f:S \rightarrow T$.
\end{proposition}

\begin{proof}
Follows directly from Propositions~\ref{prop:homotopy-equivalent} and
\ref{prop:injective}.
\end{proof}

\begin{proposition}
Let  \[
\phi_S(X_1,\ldots,X_k), \phi_T(Y_1,\ldots,Y_m), \phi_f(X_1,\ldots,X_k,Y_1,\ldots,Y_m),
\]
be first order formulas 
such that
$\RR(\phi_f,\R^k \times \R^m)$ is the graph of a semi-algebraic map
$f: S = \RR(\phi_S) \rightarrow  \RR(\phi_T) = T$.

Let
\begin{multline}
\label{eqn:formula}
				\Theta_{\phi_S,\psi_T,\phi_f}(\overline{X},\overline{Y},T) = \ \big( \ (T=0) \wedge \overline{X}=\mathbf{0} \wedge \phi_{T}(\overline{Y}) \ \big) \ \vee \\
				\big( \ (0 \leq T \leq 1) \wedge \exists \ \overline{Z} \ (\overline{X}=T \overline{Z} \wedge \phi_S(\overline{Z}) ) \wedge \phi_f(\overline{Z}, \overline{Y}) \wedge \phi_{T}(\overline{Y}) \ \big).
\end{multline}
Then, 
\[
\RR(\Theta_{\phi_S,\psi_T,\phi_f}) = \tcyl(f).
\]
Moreover,
\[
\RR(\Theta_{\phi_S,\psi_T,\phi_f} \wedge (T=1)) =  i(S) \hookrightarrow \tcyl(f).
\]
\end{proposition}

\begin{proof}
Follows directly from the definition of $\tcyl(f)$ (Eqn. \eqref{eqn:tcyl}). \end{proof}

\subsection{Replacing a closed semi-algebraic set by a semi-algebraic set defined by a closed formula}
\label{subsec:making-closed}
One basic open problem in algorithmic semi-algebraic geometry is to design an efficient algorithm which takes as input a quantifier-free formula $\phi$ such that $\RR(\phi)$ is a closed semi-algebraic set $S \subset \R^k$, and produces as output a
finite set $\mathcal{P} \subset \R[X_1,\ldots,X_k]$ and a \emph{$\mathcal{P}$-closed} formula $\psi$ such that $\RR(\phi) = \RR(\psi)$.
No algorithm with a singly exponential complexity is known for this problem.

In the absence of an efficient algorithm for solving the above problem, 
we consider the following substitute that is often enough for application.
A fundamental construction due to Gabrielov and Vorobjov \cite{GaV} gives
an efficient procedure to replace an arbitrary semi-algebraic set by a 
closed and bounded one having the same homotopy type. 
This homotopy equivalence is usually not a deformation retraction.

We describe below a construction similar to that in \cite{GaV},
when applied to a formula $\phi$  such that $\RR(\phi,B)$
is a closed semi-algebraic subset of a closed Euclidean ball $B \subset \R^k$, produces 
a closed formula $\psi$ defined over a real closed extension $\R'$ of $\R$, such that the extension of $\RR(\phi,B)$ to $\R'^k$ is a semi-algebraic deformation retraction of $\RR(\psi,B)$.

But we first need to introduce some preliminary definitions and notation.

\subsection{Real closed extensions and Puiseux series}
\label{subsec:Puiseux}
	We will need some
	properties of Puiseux series with coefficients in a real closed field. We
	refer the reader to \cite{BPRbook2} for further details.
	
	\begin{notation}
	\label{not:Puiseux}
		For $\R$ a real closed field we denote by $\R \left\langle \eps
		\right\rangle$ the real closed field of algebraic Puiseux series in $\eps$
		with coefficients in $\R$. We use the notation $\R \left\langle \eps_{1},
		\ldots, \eps_{m} \right\rangle$ to denote the real closed field 
		\[
		\R
		\left\langle \eps_{1} \right\rangle \left\langle \eps_{2} \right\rangle
		\cdots \left\langle \eps_{m} \right\rangle.
		\]
		Note that in the unique
		ordering of the field $\R \left\langle \eps_{1}, \ldots, \eps_{m}
		\right\rangle$, $0< \eps_{m} \ll \eps_{m-1} \ll \cdots \ll \eps_{1} \ll 1$.
		
	\end{notation}
	
	\begin{notation}
		\label{not:lim}
		For elements $x \in \R \left\langle \eps \right\rangle$ which are bounded
		over $\R$ we denote by $\lim_{\eps}  x$ to be the image in $\R$ under the
		usual map that sets $\eps$ to $0$ in the Puiseux series $x$.
	\end{notation}
	
	\begin{notation}
		\label{not:extension}
		If $\R'$ is a real closed extension of a real closed field $\R$, and $S
		\subset \R^{k}$ is a semi-algebraic set defined by a first-order formula
		with coefficients in $\R$, then we will denote by $\E(S, \R') \subset \R'^{k}$ the semi-algebraic subset of $\R'^{k}$ defined by
		the same formula.
		It is well known that $\E(S, \R')$ does
		not depend on the choice of the formula defining $S$ 
		\cite[Proposition 2.87]{BPRbook2}.
	\end{notation}

Let $\mathcal{P} = \{P_1,\ldots,P_s\}
\subset \R[X_1,\ldots,X_k]$ be a finite set of polynomials, and let $B \subset \R^k$ be a closed euclidean ball.

\begin{notation}
\label{not:level}
For $\sigma \in \{0,1,-1\}^\mathcal{P}$, let 
\[
\level(\sigma) = \card(\{P \in \mathcal{P} \mid \sigma(P) = 0 \}).
\]
\end{notation}	

For $c,d \in \R, 0< d < c$, and $\sigma \in \{0,1,-1\}^\mathcal{P}$, let
$\overline{\sigma}(c,d)$ denote the closed formula
\[
\bigwedge_{\sigma(P) = 0} (-d \leq P \leq d) \wedge \bigwedge_{\sigma(P) = 1} (P \geq c) \wedge \bigwedge_{\sigma(P) = -1} (P \leq -c).
\]

\begin{notation}
\label{not:Sigma-phi}
For a $\mathcal{P}$-formula $\phi$ we denote 
\[
\Sigma_{\phi} = 
\left\{ \sigma \in \{0,1,-1\}^\mathcal{P} \mid 
\left(\bigwedge_{P \in \mathcal{P}} (\mathrm{sign}(P) = \sigma(P))\right) \Rightarrow \phi \right\},
\]
where ``$\Rightarrow$''  denotes logical implication.
\end{notation}

Let 
\[
\R' = \R\la \mu_s,\nu_s, \cdots, \mu_0, \nu_0\ra = \R\la\bar\eta\ra,
\]
denoting by $\bar\eta$ the sequence $\mu_s,\nu_s, \ldots, \mu_0, \nu_0$.

\begin{notation}
\label{not:P-star}
We denote
\[
\mathcal{P}^*(\bar{\mu},\bar{\nu}) = \bigcup_{P \in \mathcal{P}} \bigcup_{j = 0}^{s} \{P \pm \mu_j, P\pm \nu_j\} \subset \R'[X_1,\ldots,X_k].
\]
\end{notation}

Finally, 
\begin{notation}
\label{not:phi-star}
We denote by
$\overline{\phi(\bar{\mu},\bar{\nu})}$ the $\mathcal{P}^*(\bar{\mu},\bar{\nu})$-\emph{closed} formula
\[
\bigvee_{\sigma \in \Sigma_{\phi}} \overline{\sigma}(\mu_{\level(\sigma)},\nu_{\level(\sigma)})
\]
(see Notation~\ref{not:Sigma-phi}).
\end{notation}

Following the notation introduced above. 
\begin{proposition}
\label{prop:closed}
Let $R > 0, B = \overline{B_k(0,R)}$, and 
suppose that $S = \RR(\phi,B)$ is closed. Then,
\[
S'  \searrow S,
\]
where $S' = \RR(\overline{\phi(\bar{\mu},\bar{\nu})},\E(B,\R'))$.
In particular, $\E(S,\R')$ is a semi-algebraic deformation retract of 
$S'$.
\end{proposition}

\begin{proof}
See \cite[Appendix]{basu-karisani}.
\end{proof}

\hide{
\begin{remark}
\label{rem:16:lem:star}
It is necessary to use multiple infinitesimals in the construction given above. As a warning consider the following example.

\begin{example}
Let $k=1,s=2, B = [-2,2]$, and
\begin{eqnarray*}
P_1 &=& X^2(X-1), \\
P_2 &=& X.
\end{eqnarray*}
Let $\sigma_1,\sigma_2$ be defined by,
$$\displaylines{
\sigma_1(P_1) =1, \sigma_2(P_2) = 1, \cr
\sigma_1(P_1) =0, \sigma_2(P_2) = 1.
}
$$
Let $\phi$ be the unique formula such that $\Sigma_{\phi} = \{\sigma_1,\sigma_2\}$.
Then, $\RR(\phi,B) = [1,2]$ is a closed semi-algebraic set, but
$\phi$ is not a closed formula.

However, if we take the closed formula $\phi^*(\mu_0,\ldots,\mu_0)$
(i.e. using only one infinitesimal)
then
\[
\lim_{\mu_0} \RR(\phi^*(\mu_0,\ldots,\mu_0),B) = \{0\} \cup [1,2] \supsetneq \RR(\phi,B).
\]
However, it is easy to verify that
\[
\RR(\phi^*(\bar\mu,\bar\nu),B) \searrow  \RR(\phi,B) = [1,2].
\]
\end{example}
\end{remark}
}

\section{Proofs of the theorems}
\label{sec:proof}
\subsection{Algorithmic Preliminaries}
\label{subsec:algo-preliminaries}

We recall the following definition from \cite{basu-karisani}.

\begin{notation} [Diagram of various unions of a finite number of subspaces]
\label{not:diagram-Delta}
Let $J$ be a finite set, $A$ a topological space, 
and $\mathcal{A} = (A_j)_{j \in J}$ a tuple of subspaces of $A$  indexed by $J$.

For any subset 
$J' \subset J$,
we denote 
\begin{eqnarray*}
\mathcal{A}^{J'} &=& \bigcup_{j' \in J'} A_{j'}, \\
\mathcal{A}_{J'} &=& \bigcap_{j' \in J'} A_{j'}, \\
\end{eqnarray*}

We consider $2^J$ as a category whose objects are elements of $2^J$, and whose only morphisms 
are given by: 
\begin{eqnarray*}
2^J(J',J'') &=& \emptyset  \mbox{ if  } J' \not\subset J'', \\
2^J(J',J'') &=& \{\iota_{J',J''}\} \mbox{  if } J' \subset J''.
\end{eqnarray*} 
We denote by $\Simp^J(\mathcal{A}):2^J \rightarrow \Top$ the functor (or the diagram) defined by
\[
\Simp^J(\mathcal{A})(J') = \mathcal{A}^{J'}, J' \in 2^J,
\]
and
$\Simp^J(\mathcal{A})(\iota_{J',J''})$ is the inclusion map $\mathcal{A}^{J'} \hookrightarrow \mathcal{A}^{J''}$.
\end{notation}

We will use an algorithm whose existence is proved in \cite[Theorem 1]{basu-karisani},
and which we will refer to as \emph{Algorithm for computing simplicial replacement},
that given a tuple of closed-formulas 
$\Phi = (\phi_0,\ldots,\phi_N)$, $R >0$,  and $\ell \geq 0$,
produces as output 
a simplicial complex $K$
and subcomplexes $K_i, 0\leq i \leq N$ of $K$, such that 
the diagram
\[
\Simp^{[N]}\left((\RR(\phi_i,\overline{B_k(0,R)}))_{i \in [N]}\right) 
\]
is homologically $\ell$-equivalent (\cite[Section 2.1.1]{basu-karisani}) 
to the diagram
\[
\Simp^{[N]}\left((|K_i|)_{i \in [N]}\right)
\]
(where $|K_i| \subset |K|$ is the geometric realization of $K_i$ and
$[N] = \{0,\ldots,N\}$).

We refer the reader to \cite{basu-karisani} for the details.

The complexity of this algorithm, as well as the size of the output
 simplicial complex $\Delta$, 
are bounded by 
\[
 (N s d)^{k^{O(m)}},
\]
where $s = \card(\mathcal{P})$, and $d = \max_{P \in \mathcal{P}} \deg(P)$.

\subsection{Proofs of Theorems~\ref{thm:functor2} and \ref{thm:functor}}
\label{subsec:proof}
We first prove Theorem~\ref{thm:functor} by describing an algorithm 
(Algorithm~\ref{alg:functor} below) and proving its correctness and analyzing its 
complexity. Theorem~\ref{thm:functor2} will then follow in 
a straightforward way.

	\begin{algorithm}[H]
		\caption{(Applying homology functor to semi-algebraic maps)}
		\label{alg:functor}
		\begin{algorithmic}[1]
			\INPUT
			\Statex{
				\begin{enumerate}[1.]
				\item $\ell \geq 0$;
				\item
a finite sets of polynomials $\mathcal{P}_S \subset \D[X_1,\ldots,X_k], ,\mathcal{P}_T \subset \D[Y_1,\ldots,Y_m], \mathcal{P}_f \subset \D[X_1,\ldots,X_k,Y_1,
\ldots Y_m]$;
                \item
                $\mathcal{P}_S$-closed formulas $\phi_S$, $\mathcal{P}_T$-closed formula $\phi_T$, 
                and $\mathcal{P}_f$-closed formula $\phi_f$,
            such that
            $\RR(\phi_S), \RR(\phi_T)$ are bounded and 
            $\RR(\phi_f,\R^k \times \R^m)$ is the graph of a semi-algebraic map $f: S = \RR(\phi_S) \rightarrow  \RR(\phi_T) = T$.
            \end{enumerate}
            }
			\OUTPUT
			\Statex{
				Simplicial complexes $\Delta_S, \Delta_T, \Delta_S \subset \Delta_T$, such that
$|\Delta_S| \hookrightarrow |\Delta_Y|$ is homologically 
$\ell$-equivalent to $f: S \rightarrow T$.
			}

\hide{
			\algstore{myalg}
		\end{algorithmic}
	\end{algorithm}
	
	\begin{algorithm}[H]
		\begin{algorithmic}[1]
			\algrestore{myalg}
}		
			
			\PROCEDURE
			\State{Let $\eps$ be an infinitesimal and $\R \gets \R\la\eps\ra$.}
			\State{Call Algorithm 14.5 in \cite{BPRbook2}  (Quantifier Elimination)  with input the existentially quantified formula 
			$\Theta_{\phi_S,\psi_T,\phi_f}$
			to obtain a quantifier-free formula $\Theta'_{\phi_S,\psi_T,\phi_f}$ equivalent to $\Theta_{\phi_S,\psi_T,\phi_f}$.}
			
			\State{ 
			$\Theta''_{\phi_S,\psi_T,\phi_f} \gets  \Theta'_{\phi_S,\psi_T,\phi_f} \wedge (T=1)$.
			}
						
			\State{
			\[
			\mathcal{P} \gets \mbox{ the set of polynomials appearing in the formula $\Theta''_{\phi_S,\psi_T,\phi_f}$}.
			\]
			}
			
			\State{Apply \emph{Algorithm for computing simplicial replacement}, with input the pair of formulas
			$(\overline{\Theta''_{\phi_S,\psi_T,\phi_f}}, \overline{\Theta'_{\phi_S,\psi_T,\phi_f}})$ 
			(recall Notation~\ref{not:phi-star})
			and $R = \frac{1}{\eps}$, and obtain as output a
			simplicial complex $K$, and subcomplexes $K_1, K_2$ of $K$.
			}
			\State{$\Delta_S \gets K_1, \Delta_T \gets K_1 \cup K_2$.}
			\State{Output the pair $(\Delta_S,\Delta_T)$.}

			\COMPLEXITY
			The complexity of the algorithm is bounded by 
\[
(s d)^{(k+m)^{O(\ell)}},
\]
where
\[
s = \max(\card(\mathcal{P}_S), \card(\mathcal{P}_T), \card(\mathcal{P}_f)),
\]
and
\[
d = \max_{P \in \mathcal{P}_S \cup \mathcal{P}_T \cup \mathcal{P}_f} \deg(P).
\]
			
		\end{algorithmic}
	\end{algorithm}
	
	\begin{proof}[Proof of correctness]
	The correctness of the algorithm follows from the correctness of 
	Algorithm 14.5 in \cite{BPRbook2} (Quantifier Elimination), Propositions~\ref{prop:closed} and
	\ref{prop:tcyl}
	 as well as the correctness of 
	Algorithm for computing simplicial replacements in \cite{basu-karisani}.
	\end{proof}
	
	\begin{proof}[Complexity analysis]
	The complexity of the algorithm follows from the complexity of
	Algorithm 14.5 in \cite{BPRbook2} (Quantifier Elimination), and the Algorithm for computing simplicial replacement.
	\end{proof}
	
	\begin{proof}[Proof of Theorem~\ref{thm:functor}]
		The theorem follows from the correctness and the complexity analysis of Algorithm~\ref{alg:functor}.
	\end{proof}
	
	\begin{proof}[Proof of Theorem~\ref{thm:functor2}]
	Theorem~\ref{thm:functor2} follows from Theorem~\ref{thm:functor} after observing that
	once we have the finite simplicial complexes $\Delta_S, \Delta_T$ with $\Delta_S \subset \Delta_T$, then using standard
	algorithms from linear algebra (Gauss-Jordan elimination) one can compute bases
	of $\HH_i(\Delta_S)$ and $\HH_i(\Delta_T), 0 \leq i \leq \ell$, and 
	the matrix for the map $\HH(j)_i$, where $j:\Delta_S \hookrightarrow \Delta_T$ is the
	inclusion map. We omit the details but it is clear that the complexity of this step is bounded polynomially in the size of $\Delta_T$. This proves
	Theorem~\ref{thm:functor2}.
	\end{proof}
	
	\section{Application to semi-algebraic zigzag persistence}	
	\label{sec:zigzag}
	In this section we discuss one application of the main result of the paper. 
	In the previous section we have designed an algorithm for effectively applying 
	the homology functor, $\HH_i(\cdot)$, to semi-algebraic maps between
	closed and bounded semi-algebraic sets. A next step  is 
	to effectively apply the homology functor to more general diagrams 
	(of semi-algebraic maps).
	
	One important class of diagrams that has been studied previously
	from an effective homology computation from the point of view were 
	diagrams of the form:
	\[
	S_0 \rightarrow S_1 \rightarrow \cdots \rightarrow S_n,
	\]
	where each $S_i$ is a closed and bounded semi-algebraic sets and all 
	maps are inclusions. This is the setting of persistent homology.
	The problem of computing the persistent homology groups (and the associated barcode)  of a filtration   of a given semi-algebraic set by the sub-level sets of 
	a semi-algebraic map was studied in \cite{basu-karisani-persistent}. It is proved in \cite{basu-karisani-persistent} that
	for each fixed $\ell > 0$, there exists an algorithm with singly exponential complexity that computes the first $\ell$-dimensional
	barcodes of such a filtration.
	
	The ideas introduced in the previous section allow us now to consider more general diagrams. Indeed 
	the notion of persistent homology has been generalized to arbitrary diagrams $D:P \rightarrow \mathbf{Top}$ (here $D$ is a functor from a 
	poset category $P$ to the category of topological spaces) \cite{Bubenik-et-al}.
	One particular class of poset diagrams that has been studied in the literature are the so called ``zigzag'' diagrams that we define below (see \cite{CarlssonSilva2010}). 
	\subsection{Zigzag diagrams}
	\label{subsec:zigzag}
	We begin by defining precisely zigzag diagrams.
	
	\begin{notation}
	\label{not:zigzag}
	We denote by $\mathbf{Z}_n$ the poset whose set of elements is  $[n]$,
	and whose Hasse diagram is indicated in the following figure.
	
\[
\xymatrix{
&1 \ar[ld] \ar[rd] && 3 \ar[ld] \ar[rd] && 5\ar[ld] \ar[rd] & \cdots& \\
0 && 2 && 4 && 6 & \cdots
}
\]
\end{notation}

\begin{definition}[Zigzag diagrams]
We call a functor $D:\mathbf{Z}_n \rightarrow \mathbf{SA}_{\R}$ from the poset 
category $\mathbf{Z}_n$ to the category of semi-algebraic sets 
and semi-algebraic maps a \emph{zigzag diagram}. 
\end{definition}

\begin{remark}
\label{rem:zigzag}
The zigzag diagrams that we consider in this paper where the arrows (maps) alternate in directions are not the most general possible. A general zigzag diagram need not have this alternation. Applying the homology functor to general zigzag diagrams gives rise to 
zigzag persistence modules which are precisely the quiver representations of 
quivers of type $\mathbf{A}$ (see \cite{CarlssonSilva2010}). We restrict to the case of alternating arrows (also called \emph{regular zigzag diagrams} for the ease of exposition and simplifying notation.
\end{remark}

We prove below that
for each fixed $\ell \geq 0$, there exists an
algorithm that given a zigzag functor 
$D:\mathbf{Z}_n \rightarrow \mathbf{SA}_{\R}$
(i.e. given quantifier-free formulas describing the semi-algebraic sets 
$D(i)$ and the graphs of the various maps $D(i) \rightarrow D(i-1), D(i) \rightarrow D(i+1)$ for every odd $i, 0\leq i \leq n$), computes 
a functor $D'$ from $\mathbf{Z}_n$ to the category of finite simplicial complexes,
such that the composition of $D'$ with the geometric realization
functor $|\cdot|$ is homologically $\ell$-equivalent to $D$, and all the morphisms
$D'(i) \rightarrow D'(i-1), D'(i) \rightarrow D'(i+1)$ are inclusions. Moreover, the complexity of the algorithm is bounded by $(n s d)^{k^{O(\ell)}}$,
where $s$ is the cardinality of the set of polynomials occurring in all the formulas in the input and $d$ their maximum degree.

The more precise statement is as follows.

\begin{theorem}
\label{thm:zigzag}
    For each fixed $\ell \geq 0$, there exists an algorithm with the
    following properties. The algorithm takes the following input:
    \begin{enumerate}[(a)]
    \item $R >0$;
        \item 
        a tuple of closed formulas
    $\Phi = (\phi_0,\ldots,\phi_n)$, with $S_i = \RR(\phi_i,B) \subset \R^k$ for $0 \leq i \leq n$, where $B = \overline{B_k(0,R)}$;
    \item
    a tuple of 
    closed formulas 
    $
    \Psi = (\psi_1,\ldots,\psi_n),
    $
    such that $\RR(\Psi_i, B \times B)$ is the graph of 
    a semi-algebraically continuous map $f_i:S_i \rightarrow S_{i-1}$
    if $i$ is odd, and is the graph of 
    a semi-algebraically continuous map $f_i:S_{i-1} \rightarrow S_{i}$ if $i$ is even.
    \end{enumerate}
    The algorithm produces as output simplicial complexes
    $\Delta_0,\ldots,\Delta_n$, having the following properties:
    \begin{enumerate}[(a)]
        \item 
        $\Delta_i$ is a subcomplex of $\Delta_{i-1}$ and  $\Delta_{i+1}$ if $i$ is odd, and $\Delta_{i-1},\Delta_{i+1}$ are subcomplexes of $\Delta_i$
    if $i$ is even.
        \item
        The diagram
    \[
    |\Delta_0| \hookleftarrow |\Delta_1| \hookrightarrow |\Delta_2| \cdots
    \]
    is homologically $\ell$-eqivalent to
    \[
    S_0 \leftarrow S_1 \rightarrow S_2 \cdots
    \]
    \end{enumerate}
     
    The complexity of the algorithm is bounded by $(n s d)^{k^{O(\ell)}}$,
where $s$ is the cardinality of the set of polynomials occurring in all the formulas in the input and $d$ their maximum degree.
\end{theorem}

\subsection{Proofs of Theorems~\ref{thm:zigzag-persistence}, \ref{thm:zigzag2} and \ref{thm:zigzag}}
We first prove Theorem~\ref{thm:zigzag} which is the main theorem of this section.
Theorems~\ref{thm:zigzag-persistence} are \ref{thm:zigzag2} straightforward
consequences.

In order to prove Theorem~\ref{thm:zigzag}, we first introduce a more general
mapping cylinder construction which is adapted to the zigzag diagrams. For 
a zigzag diagram consisting of just one zigzag 
\[
\xymatrix{
S_{i-1} \ar[rd]^{f_i} && S_{i+1}\ar[ld]_{f_{i+1}} \\
&S_i&
}
\]
we would like to replace the diagram by the union of two mapping cylinders glued
along 
$S_i$
as shown in Figure~\ref{fig:cyl}.
The maps $f_{i}, f_{i+1}$ are replaced by inclusions of $S_{i-1}$ and $S_{i+1}$ into the mapping cylinders of $f_i$ and $f_{i+1}$, and $S_i$ is replaced by the union of the mapping cylinders of $f_i$ and $f_{i+1}$ (oriented in opposite directions and glued along $S_i$).

\begin{figure}[H]
	\begin{center}
		\includegraphics[scale=.8]{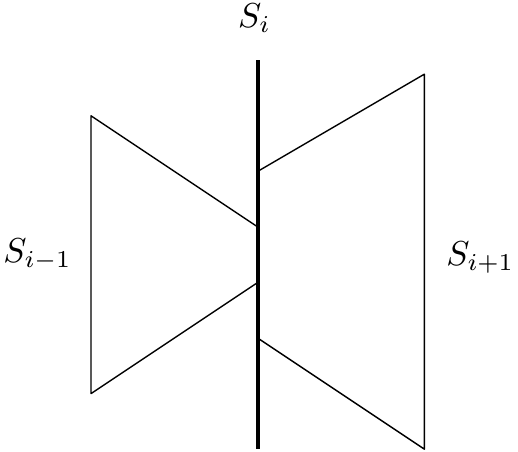}
		\caption{\small }
		\label{fig:cyl}
	\end{center}
\end{figure}

We now extend the above idea in two directions. We consider zigzag diagrams with a finite number of zigzags instead of just one,  
and instead of the classical mapping cylinder 
we use the semi-algebraic version introduced in Definition~\ref{def:cylinder}.

We first define directed versions of the semi-algebraic mapping cylinder construction.

\begin{definition}[Directed semi-algebraic mapping cylinder]
\label{def:cylinder-directed}
    Let $S \subset \R^k, T \subset \R^k$ be semi-algebraic subsets and 
    $f:S \rightarrow T$ a semi-algebraic map, and $a,b \in \R,  a < b$.
    We denote:
    \begin{eqnarray*}
    \label{eqn:cyl-modified}
				\underset{\rightarrow}{\tcyl}(f)(a,b) &=& \{ ((\lambda - a)/(b-a) \cdot x, f(x), \lambda) \ | \ x \in S, \lambda \in [a, b]\} \cup \\
				&&\{ (\mathbf{0}, y, a) \ | \ y \in T\},\\
				\underset{\leftarrow}{\tcyl}(f)(a,b) &=& \{ ((b -\lambda)/(b-a) \cdot x, f(x), \lambda) \ | \ x \in S, \lambda \in [a, b]\} \cup \\
				&&\{ (\mathbf{0}, y, b) \ | \ y \in T\}.
	\end{eqnarray*}
(Note that $\tcyl(f) = \underset{\rightarrow}{\tcyl}(f)(0,1)$.) 
\end{definition}

We now define semi-algebraic mapping cylinders of zigzag diagrams.
\begin{definition}[Semi-algebraic mapping cylinder of zigzag diagrams]
\label{def:cylinder-general}
Let $D:\mathbf{Z}_n \rightarrow \mathbf{SA}_{\R}$ be a zigzag diagram and 
for $0 \leq i \leq n$, let $S_i = D(i)$, and for $1 \leq i \leq n$, let $f_i$
denote the map
$f_i: S_{i-1} \rightarrow S_i$.

For $0 <  i  \leq n$, with $i$ odd we define
\[
\widetilde{S}_i = \begin{cases}
\{ (x, h_i(x,\mu), \mu) \mid x \in S_i, \mu \in [i - \tfrac{1}{2}, i+ \tfrac{1}{2}]\}, &i\neq n, \\
\{ (x, (\mu - n + \tfrac{3}{2}) \cdot f_{i}(x), \mu) \mid x \in S_n, \mu \in [n - \tfrac{1}{2}, n]\} & i = n,
\end{cases}
\]
where
\[
h_i(x,\mu) = (\mu -i + \tfrac{3}{2}) \cdot f_{i}(x) + (\mu -i +\tfrac{1}{2}) \cdot f_{i+1}(x).
\]

For $0 \leq  i  \leq n$, with $i$ even,  we define
\[
\widetilde{S}_i = \begin{cases}
\underset{\leftarrow}{\tcyl}(f_{i})(i - \tfrac{1}{2},i) \cup \underset{\rightarrow}{\tcyl}(f_{i+1})(i,i+\tfrac{1}{2}) \cup \widetilde{S}_{i-1} \cup \widetilde{S}_{i+1}
& i \neq 0,n, \\
\underset{\rightarrow}{\tcyl}(f_{1})(0,\tfrac{1}{2}) \cup  \widetilde{S}_{1} & 
i=0, \\
\underset{\leftarrow}{\tcyl}(f_{n})(n - \tfrac{1}{2},n)  \cup \widetilde{S}_{n-1}
& i=n.
\end{cases}
\]

We denote the diagram 
    \[
\xymatrix{
&\widetilde{S}_1 \ar[ld] \ar[rd] && \widetilde{S}_3 \ar[ld] \ar[rd] && \widetilde{S}_5\ar[ld] \ar[rd] & \cdots& \\
\widetilde{S}_0 && \widetilde{S}_2 && \widetilde{S}_4 && \widetilde{S}_6 & \cdots
}
\]
where the arrows denote inclusions, 
by $\tcyl(D)$.
\end{definition}

Suppose in Definition~\ref{def:cylinder-general} each $S_i$ is a closed and bounded semi-algebraic subset of $\R^k$. 
Notice the following (see also the schematic Figure~\ref{fig:cyl-general}).
\begin{enumerate}[(a)]
\item For each even $i$,  $\widetilde{S}_{i-1}, \widetilde{S}_{i+1} \subset \widetilde{S}_i$
(using the convention that $\widetilde{S}_{-1} = \widetilde{S}_{n+1} = \emptyset$).
\item For each odd $i$, the map $\widetilde{S}_i \rightarrow S_i, (x,y,\mu) \mapsto x$ is
a homological equivalence using the homological version of the Vietoris-Begle theorem.
\item
The union $\bigcup_{0 \leq i \leq n} \widetilde{S}_i$ is a closed and bounded semi-algebraic subset of
$\R^k \times \R^k \times [0,n]$.
\item For $T \subset \R^k \times \R^k \times\R$, and $\mu \in \R$,  
let $T_\mu$ denote the subset of $T$ with last coordinate equal to $\mu$. Then
for each even $i$, we have 
\begin{eqnarray*}
\left(\widetilde{S}_{i-1}\right)_{i-\tfrac{1}{2}} &=& \left(\underset{\leftarrow}{\tcyl}(f_i)(i-\tfrac{1}{2},i)\right)_{i - \tfrac{1}{2}}, \\
\left(\widetilde{S}_{i+1}\right)_{i+\tfrac{1}{2}} &=& \left(\underset{\rightarrow}{\tcyl}(f_{i+1})(i-\tfrac{1}{2},i)\right)_{i + \tfrac{1}{2}}, \\
\left(\widetilde{S}_i\right)_i &=& \mathbf{0} \times S_i \times \{i\} \cong S_i.
\end{eqnarray*}
\end{enumerate}

\begin{figure}[H]
	\begin{center}
		\includegraphics[scale=.9]{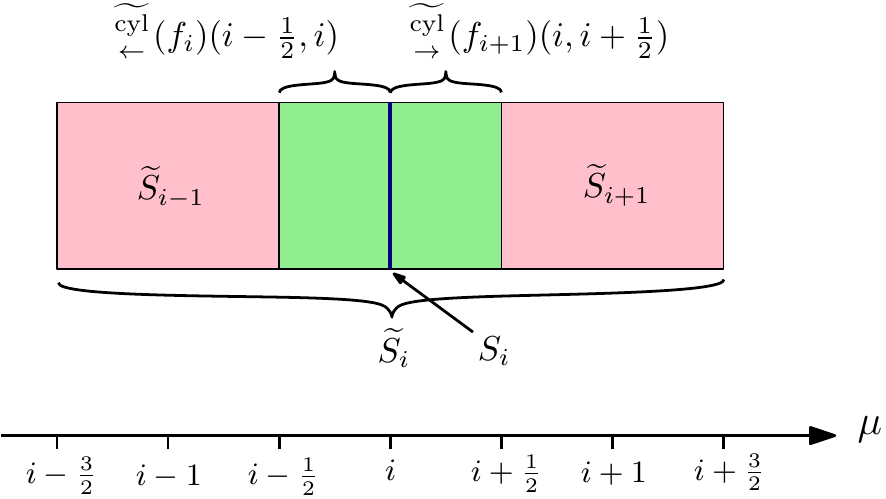}
		\caption{\small $\widetilde{S}_{i-1}, \widetilde{S}_i, \widetilde{S}_{i+1}$ for every even $i$.}
		\label{fig:cyl-general}
	\end{center}
\end{figure}

The following proposition captures the key property of $\tcyl(D)$. 
We use the same notation introduced in
Definition~\ref{def:cylinder-general},

\begin{proposition}
\label{prop:zigzag}
Let $D:\mathbf{Z}_n \rightarrow \mathbf{SA}_{\R}$ be a zigzag diagram and 
for $0 \leq i \leq n$  $S_i = D(i)$ is a closed and bounded semi-algebraic subset of $\R^k$.
Then the diagrams $D$ and $\tcyl(D)$ are homologically equivalent. 
\end{proposition}

\begin{proof}
For $0 \leq j \leq n$, we define  $g_j: \widetilde{S}_j\rightarrow {S}_j$ be defined as follows.

\[
g_j(x,y,\mu) = \begin{cases}
y & \text{ if } j \text{ is even, }  j-\tfrac{1}{2} \leq \mu \leq j+\tfrac{1}{2}, \\ 
f_{j}(x)& \text{ if } j \text{ is even, } j\neq 0,  j-\tfrac{3}{2} \leq \mu \leq j-\tfrac{1}{2}, \\ 
f_{j+1}(x)& \text{ if } j \text{ is even, } j\neq n,  j+ \tfrac{1}{2} \leq \mu \leq j+\tfrac{3}{2}, \\ 
x& \text{ if } j \text{ is odd. }
\end{cases}
\]

It is now easy to check that the for each  (even) $j$ the following diagram commutes:
\[
\xymatrix{
\widetilde{S}_{j-1}\ar[dd]^{g_{j-1}} \ar[rd] && \widetilde{S}_{j+1}\ar[dd]^{g_{j+1}} \ar[ld] \\
& \widetilde{S}_{j}\ar[dd]^{g_j} & \\
{S}_{j-1} \ar[rd]^{f_{j-1}}&& {S}_{j+1}\ar[ld]^{f_j} \\
& {S}_{j} & \\
}
\]

Finally, using the homological version of the Vietoris-Begle theorem it is an easy exercise to check that each $g_j, 0 \leq j \leq n$ is a homological equivalence.
This proves the proposition.
\end{proof}

\begin{proof}[Proof of Theorem~\ref{thm:zigzag}]
Let $D$ denote the zigzag diagram 
\[
\xymatrix{
&{S}_1 \ar[ld]^{f_1} \ar[rd]_{f_2} && {S}_3 \ar[ld]^{f_3} \ar[rd]_{f_4} && {S}_5\ar[ld]^{f_5} \ar[rd]_{f_6} & \cdots& \\
{S}_0 && {S}_2 && {S}_4 && {S}_6 & \cdots
}
.
\]
Using Algorithm 14.5 in \cite{BPRbook2} (Quantifier Elimination) and following
Definition~\ref{def:cylinder-general} we can compute using the input tuples
of formulas $\Phi$ and $\Psi$ 
a tuple $\widetilde{\Phi} = (\widetilde{\phi}_0,\ldots,\widetilde{\phi}_n)$ 
whose realization is $\tcyl(D)$. Finally we replace the tuple of formulas
$\widetilde{\Phi}$ by a tuple of closed formulas
\[
\overline{\widetilde{\Phi}}= (\overline{\widetilde{\phi}_0},\ldots,\overline{\widetilde{\phi}_n})
\]
(recall Notation~\ref{not:phi-star}).
The number of polynomials and their degrees appearing in the
$\overline{\widetilde{\Phi}}$
are all bounded singly exponentially. We then use the Algorithm for simplicial replacement
to compute the simplicial complexes, $\Delta, \Delta_0,\ldots,\Delta_n$ having the required
properties.
\end{proof}

\begin{proof}[Proof of Theorem~\ref{thm:zigzag2}]
Theorem~\ref{thm:zigzag2} follows from Theorem~\ref{thm:zigzag}
using standard algorithms from linear algebra.
\end{proof}

\begin{proof}[Proof of Theorem~\ref{thm:zigzag-persistence}]
Use Theorem~\ref{thm:zigzag2} to reduce the computation of the barcode of a zigzag diagram
	of semi-algebraic maps to that of finite dimensional vector spaces, and
	then use the algorithm in \cite[Section 4.2]{CarlssonSilva2010}.
	The complexity remains singly exponentially bounded.
\end{proof}

\section{Conclusion and open problems}
\label{sec:conclusion}
In this paper we have described for each $\ell \geq 0$, algorithms with singly exponential complexity for computing the homology functor $\HH_i(\cdot)$ for $0 \leq i \leq \ell$
on semi-algebraic maps and zigzag diagrams
on closed and bounded semi-algebraic sets.

We end with some open problems. 
\begin{enumerate}[1.]
    \item Is it possible to compute the homology functor with singly exponential complexity without having to restrict to only the first few dimensions ?
    
    \item Remove the  assumption that all the semi-algebraic sets in the input are closed and bounded. One possible approach is to generalize the results of Gabrielov and Vorobjov \cite{GaV} on replacing an arbitrary semi-algebraic set by a locally closed one without changing its  homotopy type, to semi-algebraic maps.
    
    \item Is it possible to extend the numerical algorithms mentioned in the Introduction  for computing Betti numbers of semi-algebraic sets to the functor setting ? It will be necessary to study the  condition number of semi-algebraic maps or more general diagrams ?
    
    \item Study the categorical complexity 
    of the semi-algebraic homology functor, and prove a singly exponential bound on its
    functor complexity (as defined in \cite{Basu-Isik}).
\end{enumerate}

	\bibliographystyle{amsplain}
	\bibliography{master}
\end{document}